\documentclass[12pt]{amsart}
\pdfoutput=1
\usepackage[foot]{amsaddr}
\usepackage{graphicx}
\usepackage{color}
\usepackage{xifthen}
\newboolean{pedro}
\setboolean{pedro}{true}
\newcommand{\pedro}{\ifthenelse{\boolean{pedro}}{\color{magenta}
    \setboolean{pedro}{false}}{\color{black}\setboolean{pedro}{true}}}
\newcounter{margin}

\usepackage{tikz}
\usepackage{fp}
\usetikzlibrary{fixedpointarithmetic}
\usetikzlibrary{math}
\usetikzlibrary{fit}

\usepackage{amsmath,amssymb,hyperref,srcltx}

\DeclareMathOperator{\Inv}{Inv}
\DeclareMathOperator{\NI}{NI}

\newboolean{javier}
\setboolean{javier}{true}
\newcommand{\javier}{\ifthenelse{\boolean{javier}}{\color{red}
    \setboolean{javier}{false}}{\color{black}\setboolean{javier}{true}}}

\newtheorem{pro}{Proposition}[section]
\newtheorem{teo}{Theorem}[section]
\newtheorem*{main}{Main Theorem}
\newtheorem{cor}{Corollary}[section]

\theoremstyle{remark}
\newtheorem{rem}{Remark}[section]

\theoremstyle{definition}
\newtheorem{defi}{Definition}[section]
\date{\today}
\begin{document}

\subjclass[2010]{32S05, 32S65, 14H20}
\keywords{ holomorphic foliation, Poincar\'{e} problem, plane branch, Puiseux parametrization}

\title{The local Poincar\'e problem for irreducible branches}


\author{J.Cano}
\address{Dpt. de \'{A}lgebra, An\'{a}lisis, etc. Univ. de Valladolid}
\author{P. Fortuny Ayuso}
\address{Dpt. de Matem\'{a}ticas, Univ. de Oviedo}
\author{and J. Rib\'{o}n}
\address{Dpt. de An\'{a}lise, Univ. Federal Fluminense}
\thanks{The authors are partially supported
  by MTM2016-77642-C2-1-P (AEI/FEDER, UE)}

\begin{abstract}
Let ${\mathcal F}$ be a germ of holomorphic foliation defined in a neighborhood of the origin of
${\mathbb C}^{2}$ that has a germ of irreducible holomorphic invariant curve $\gamma$. 
We provide a lower bound for the vanishing multiplicity of ${\mathcal F}$ at the origin 
in terms of the equisingularity class of $\gamma$. Moreover, we show that such a lower bound is
sharp. Finally, we characterize the types of dicritical singularities for which the multiplicity of $\mathcal{F}$ can be bounded in terms of that of $\gamma$ and provide an explicit bound in this case.
\end{abstract}

\maketitle

\bibliographystyle{plain}


\section{Introduction}

The Poincar\'{e} problem consists in determining whether or not an algebraic
differential equation in two variables has a rational first integral and then
calculating such a function. This problem is strongly related with the problem of
finding upper bounds for the degree of an invariant variety in terms of the degree of
the foliation and other invariants. Indeed Poincar\'{e} noticed that the existence of
rational first integrals is algebraically decidable if an upper bound for the degree
çof a generic algebraic leaf is provided.

We study in this paper the local version of the Poincar\'{e} problem: consider a germ
of holomorphic foliation ${\mathcal F}$ defined in a neighborhood of the origin of
${\mathbb C}^{2}$. Let $\gamma$ be a plane branch (i.e. a germ of irreducible
analytic curve at the origin) that is invariant by ${\mathcal F}$.  We
obtain lower bounds for the vanishing multiplicity $\nu_0 ({\mathcal F})$ (also
order) of ${\mathcal F}$ (cf. Definition \ref{def:order}) in terms of geometrical
data of $\gamma$. The lower bound provided by the  Main Theorem 
does not depend on the  foliation, not  even on its desingularization.
Other results (Proposition \ref{pro:calam}) give improved lower bounds of 
$\nu_0 ({\mathcal F})$ in terms of 
the multiplicity (order) $\nu_0 (\gamma)$ (cf. Definition \ref{def:order}) at the origin and
combinatorial data of ${\mathcal F}$.

There is a strong connection between
the local and the global Poincar\'{e} problems. This is an important ingredient in
the work of Cerveau and Lins Neto \cite{CL:iac} where they show that an invariant
invariant curve $C$ of a foliation ${\mathcal F}$ in ${\mathbb C}{\mathbb P} (2)$
satisfies $\deg (C) \leq \deg ({\mathcal F})+2$ if all singularities of $C$ are
nodal, i.e. normal crossings singularities. The same formula was proved by Carnicer
for the case in which ${\mathcal F}$ has no dicritical singularities on $C$ without
restrictions on the singularities of $C$ \cite{car}.  Du Plessis and Wall show that
the degree of an invariant curve $C$ can be bounded from above by a function of the
degree of the foliation ${\mathcal F}$ and local equisingular data of $C$ (for
instance, the sum of the Milnor numbers of $C$ at its singular points)
\cite{Plessis-Wall}. Recently, several papers have appeared in which the point of
view of local polar invariants has been used to study the Poincar\'{e} problem
(Corral and Fern\'{a}ndez-S\'{a}nchez) \cite{Corral-Fernandez:isolated}, (F. Cano,
Corral and Mol) \cite{Cano-Corral-Mol:polar} and (Genzmer and Mol)
\cite{Genzmer-Mol:polar}.

The global Poincar\'{e} problem is an active subject of research. Some recent
contributions have been provided by Brunella \cite{Brunella:indices}, Campillo and
Carnicer \cite{Campillo-Carnicer:Poincare}, Campillo and Olivares
\cite{Campillo-Olivares:base}, Soares \cite{Soares:annals}, Pereira
\cite{Pereira:Poincare}, Esteves and Kleiman \cite{Esteves-Kleiman:bounds}, Cavalier
and Lehmann \cite{Cavalier-Lehmann:inequality}, Galindo and Monserrat
\cite{Galindo-Monserrat:dicritical}...

In the local setting, the vanishing orders satisfy
$\nu_0 ({\mathcal F}) \geq \nu_0 (\gamma) -1$ if the foliation ${\mathcal F}$ is
non-dicritical, i.e. has finitely many invariant branches, by a theorem of Camacho,
Lins Neto and Sad \cite{Camacho-Lins-Sad:topological}.  Indeed, the inequality
becomes an equality $\nu_0 ({\mathcal F}) = \nu_0 (\Gamma) -1$ when $\Gamma$ is the
union of all the invariant branches through the origin and ${\mathcal F}$ is a
generalized curve, i.e. there are no saddle-nodes singularities in its
desingularization process \cite{Camacho-Lins-Sad:topological}.

A strict statement of the local Poincar\'{e} problem might be: given a branch $\gamma$
that is invariant by a germ of holomorphic foliation ${\mathcal F}$, find a function (if it exists) $h: {\mathbb N} \to {\mathbb N} \cup \{0\}$ such that
\begin{equation}
\label{equ:goal}
\nu_0 ({\mathcal F}) \geq h (\nu_0 (\gamma)) \ \mathrm{with} \ \lim_{m \to \infty} h(m)=\infty.
\end{equation}
This is equivalent to the existence of a function
$g:{\mathbb N} \cup \{0\} \to {\mathbb N} $ such that
$\nu_0 (\gamma) \leq g (\nu_0 ({\mathcal F}))$.  Indeed, given $h$ we can define
$$g(m) =  \max \{ n \in {\mathbb N}:  h(n) \leq m \}  . $$
For instance $h(m)=m-1$ is such a function if ${\mathcal F}$ is non-dicritical. Such
a formula does not exist in general: given $p, q \in {\mathbb N}$ with
$\gcd (p,q)=1$, the curve $y^{p} - x^{q}=0$ is an invariant curve of order
$\min (p,q)$ that is invariant by the foliation $p x dy - q y dx=0$ of order $1$.  So
it is not possible to obtain $h$ satisfying Equation (\ref{equ:goal}) even if we
replace $\nu_0 (\gamma)$ by the Milnor number of $\gamma$.  This kind of pathology
was described in the global case by Lins Neto.  He even exhibits examples of
one-parameter families of foliations of fixed degree such that neither the degree of
the generic leaf nor its genus is bounded among the foliations in the family with
rational first integral \cite{LN:Poin}.

However, as the topology of $\gamma$ is given by its Puiseux characteristics, it is
natural to try to obtain lower bounds for $\nu_0 ({\mathcal F})$ in terms of the
equisingularity class of $\gamma$.  Indeed, our main result is a formula of the same
type as Equation (\ref{equ:goal}) in which the right hand side depends on the
topological class of $\gamma$. More precisely,
we provide a solution of the local Poincar\'{e} problem for irreducible curves
in which the multiplicity of the curve
is replaced with another topological and equisingular invariant:
the next to last partial multiplicity of $\gamma$.

Consider an injective Puiseux parametrization
$\gamma (t) =  (t^{n} , \sum_{j = n}^{\infty} a_j t^j)$ of $\gamma$ and let
$m_j = \gcd (\{n \} \cup \{ n \leq k \leq j : a_k \neq 0 \})$.
The cardinal of the finite set  ${\mathcal M}:=\{ n/m_j : j \geq n\} \setminus \{1\}$ is
by definition the genus $g$ of $\gamma$; it is equal to $0$ if and only if $\gamma$ is regular.
We can order the elements of ${\mathcal M}$ in a sequence
$$q_0 = 1 < q_1 < q_2 < \hdots < q_g = n .$$
The last element $q_g$ is the multiplicity $n$ of $\gamma$ at the
origin, whereas $q_0, q_1, \hdots, q_g$ are by definition the ``partial multiplicities" of $\gamma$.
We define the {\it virtual multiplicity} $\mu (\gamma)$ as the next to last partial
multiplicity $q_{g-1}$.
\begin{main}
Let ${\mathcal F}$ be a germ of singular holomorphic foliation, defined in a
neighborhood of the origin in ${\mathbb C}^{2}$,
with a singular invariant branch $\gamma$. Then we have
$\nu_0 ({\mathcal F}) \geq \mu (\gamma) = q_{g-1} \geq 2^{g-1}$.
\end{main}
The lower bound $\mu(\gamma)$ for $\nu_0 ({\mathcal F})$ is sharp (Remark \ref{rem:sharp}).
The next to last partial multiplicity $q_{g-1}$ is equal to $1$ for the curve $y^{p} = x^{q}$.
The Main Theorem is very practical to obtain families of curves such that Property
(\ref{equ:goal}) is satisfied: consider the curves with Puiseux parametrization
$$ \gamma_n (t) = (t^{30 n}, t^{30n +30} + t^{30n +45} + t^{30n + 55 } + t^{30n + 56}) $$
for $n \geq 2$. We have $g=4$ and $q_1= n$, $q_2=2n$, $q_3=6n$ and $q_4 = 30n$.
As a consequence, $\nu_0 ({\mathcal F}) \geq 6n = \frac{\nu_0 (\gamma_n)}{5}$
for any germ of foliation that preserves $\gamma_n$.

The lower estimate for $\nu_0 ({\mathcal F})$ provided by the Main Theorem depends
only on the equisingularity class of $\gamma$. If, on top of that, we know the list
of irreducible components of the divisor of the desingularization of $\gamma$ which
are invariant for $\mathcal{F}$, then we give a better lower estimate for
$\nu_0 ({\mathcal F})$ in terms of all partial multiplicities of $\gamma$
(Proposition \ref{pro:calam} and in particular Equation (\ref{equ:calam2})). This
formula validates the study of the local Poincar\'{e} problem via Puiseux
characteristics of invariant branches.

The main ingredient of the proof of $\nu_0 ({\mathcal F}) \geq \nu_0 (\gamma) -1$ in
the non-dicritical case is an index formula \cite{Camacho-Lins-Sad:topological}
relating $\nu_0 ({\mathcal F})$ with some vanishing indices that appear along the
desingularization process.  The index formula admits a generalization for any kind of
foliation, as was discovered by Hertling \cite{Hertling:formules}. This formula contains two major terms, an ``index term'' and a term that depends on the
combinatorics of the desingularization process.  It is not difficult to give lower
bounds for either the index part or the combinatorics part, but we achieve to relate
the properties of both parts of Hertling's formula to obtain better lower
estimates of $\nu_0 ({\mathcal F})$. For instance, we can show that the existence of
non-invariant irreducible components of the divisor of the desingularization process
worsens the lower estimates of $\nu_0 ({\mathcal F})$ only in very specific cases
that can be completely characterized (Proposition \ref{pro:minlam}).  In particular
we can classify the case in which the worst lower estimates are obtained (Theorem
\ref{teo:bound-H-all-denoms}).

Our techniques can be used to obtain results for the case of reducible germs of invariant curves 
and for the global Poincar\'{e} problem. These subjects will be pursued in future works.
\section{The index formula}

Let $\gamma$ be a germ of irreducible analytic curve  that is
invariant by a germ of local holomorphic foliation ${\mathcal F}$ defined in a neighborhood of
$0$ in ${\mathbb C}^2$. 
Before looking into the joint behaviour of the invariants of $\gamma$ and $\mathcal{F}$,
we introduce Hertling's formula \cite{Hertling:formules}
that relates the multipicity of the singularity with indexes of the foliation
localized at points of the divisor of a desingularization process.
The setting of this section is the following: we consider a neighbourhood of a point $P$ belonging to a (germ of) $2$-dimensional analytic complex manifold $M$, endowed with a map $\pi:M\rightarrow (\mathbb{C}^2, 0)$, which is a composition of blow-ups (right now we do not need to specify its structure). We call $\mathcal{F}'$ the strict transform of $\mathcal{F}$ by $\pi$. The point $P$ belongs to the exceptional divisor $P\in E = \pi^{-1}(0,0)$. The main invariants we shall need are given by the following definitions:
\begin{defi} 
\label{def:order}
Let $f \in {\mathbb C} \{x_1,\hdots,x_n\} \setminus \{0\}$ be a convergent power series with
complex coefficients  in $n$ variables. We say that $\nu_0 (f)$ is the {\it (vanishing) order} of $f$
at $0$ (or also its {\it (vanishing) multiplicity}) if
$f \in {\mathfrak m}^{\nu_0 (f)} \setminus  {\mathfrak m}^{\nu_0 (f)+1}$ where
${\mathfrak m}$ is the maximal ideal of  ${\mathbb C}  \{x_1,\hdots,x_n\} $.

Let $X= \sum_{i=1}^{n} a_i (x_1, \hdots, x_n) \partial / \partial_{x_n}$ be a germ
of holomorphic vector field. We define its order at $0$ as
$\nu_0 (X)= \min_{1 \leq i \leq n} \nu_0 (a_i)$.  Given a germ of holomorphic foliation
${\mathcal F}$ at $({\mathbb C}^2,0)$, whose leaves are the trajectories of a vector field
$X=a(x,y)\partial/\partial x + b(x,y)\partial/\partial y$ (with $\gcd (a,b)=1$),
we define $\nu_0 ({\mathcal F})= \nu_0 (X)$.

Given a germ of holomorphic curve $\gamma$ at $({\mathbb C}^2,0)$, we define
$\nu_0 (\gamma)= \nu_0 (f)$ where $(f)\subset\mathbb{C}\{x,y\}$ is the ideal
$I(\gamma)$ of analytic functions whose restriction to $\gamma$ is the zero
function. The vanishing order of $g \in {\mathbb C}\{x,y\}$ along $\gamma$ at $0$ is
the unique $\nu \in {\mathbb N} \cup \{0\}$ such that
$g \in I(\gamma)^{\nu} \setminus I(\gamma)^{\nu +1}$.

\end{defi}
\begin{defi}\label{def:weight}  
Let $P$ be a point of an irreducible component $D$ of the exceptional divisor $E$ and
$g  \in {\mathbb C}\{x,y\}$. We define the vanishing order of $g \circ \pi$ along $D$
as the vanishing order of $g \circ \pi$ along $D$ at $P$
(it is clearly independent of the choice of $P$). 
The {\it weight} $w (D)$ of an irreducible component of the exceptional divisor $E$ is the
vanishing order of $l \circ \pi$ along $D$ for a generic linear map $l: {\mathbb C}^{2} \to {\mathbb C}$.
\end{defi}

\begin{rem}
The weight of the divisor of the blow-up of the origin is equal to $1$.
Consider a point $P$ of the {exceptional} divisor of a sequence of blow-ups $\pi$.
If $P$ is not a corner point (i.e. it does not belong to two irreducible components of $E$) and belongs to the irreducible {component} $D_1$, then the weight of the irreducible component $\overline{D}$ corresponding to the blow-up of $P$ is $w (\overline{D})=w (D_1)$. If, on the contrary, $P$ is a corner point belonging to two irreducible components $D_1$ and $D_2$ of $E$, then the weight $w(\overline{D})$ of the new irreducible component of the divisor $\overline{D}$ is $w(\overline{D})=w(D_1)+w(D_2)$.
\end{rem}

Fix a system of local coordinates $(x,y)$ at $P$ and consider a vector field $X=a(x,y)\partial/\partial x + b(x,y)\partial/\partial y$ whose corresponding foliation is $\mathcal{F}$, with $\gcd(a,b)=1$. Let $\eta:(\mathbb{C},0)\rightarrow M$ be an analytic curve through $P$, which \emph{in this section} will always be an irreducible component $D$ of $E$ (this the reason we use $\eta$ instead of $\gamma$).

\begin{defi}[See \cite{Camacho-Lins-Sad:topological}]\label{def:aleph}
The vanishing order $\aleph_{P} ({\mathcal F}, \eta)$ of
${\mathcal F}$ along and invariant curve $\eta$ at $P$
is the order of the holomorphic vector field $\eta^{*} X$ at $0$.
\end{defi}
\begin{rem}
The index $\aleph_{P} ({\mathcal F}, \eta)$ coincides with the GSV
(G\'{o}mez Mont-Seade-Verjovsky) index when $\eta$ is smooth \cite{zbMATH04196386}.
\end{rem}
\begin{defi}[Ibid.]
\label{def:kapinv}
Assume $P$ is a singular point of ${\mathcal F}'$ and $D_1$ is an invariant irreducible component of $E$ with $P \in D_1$. We define $\kappa_{P} ({\mathcal F}', D_1)= \aleph_{P}({\mathcal F}', D_1)-1$ if $P$ is a corner point and the other irreducible component $D_2$ of the divisor $\pi^{-1}(0,0)$
through $P$ is also ${\mathcal F}'$-invariant and
$\kappa_{P} ({\mathcal F}', D_1)= \aleph_{P}({\mathcal F}', D_1)$ otherwise.
\end{defi}
\begin{defi}[Ibid.]
Assume $\eta$ is not invariant by ${\mathcal F}$. If $f=0$ is an irreducible equation of $\eta$, we define $\mathrm{tang}_{0}({\mathcal F},\eta)$ as the dimension of the complex vector space ${\mathcal O}_{2}/(f, X(f))$.
\end{defi}

\begin{defi}
Assume $P$ is a singular point of ${\mathcal F}'$ and $D_1$ is a non-invariant irreducible component
of $E$ with $P \in D_1$, 
we define $\kappa_{P}  ({\mathcal F}', D_1) =\mathrm{tang}_{P}({\mathcal F}', D_1)$. 
\end{defi}

\begin{rem}\label{rem:order-formula-for-non-dicritical}
The following formula was established in \cite{Camacho-Lins-Sad:topological} for a non-dicritical foliation $\mathcal{F}$ at $(\mathbb{C}^2,0)$ and any sequence $\pi$ of point blow-ups:
\begin{equation}
\label{equ:funfor}
\sum_{D_j} \sum_{P \in D_j} w(D_j) \kappa_P ({\mathcal F}', D_j) = \nu_{0} ({\mathcal F}) + 1,
\end{equation}
where $D_1, D_2, \hdots$ are the irreducible components of the divisor of $\pi$, ${\mathcal F}'$ is the strict transform of ${\mathcal F}$ by $\pi$ and
$\nu_{0} ({\mathcal F})$ is the multiplicity of ${\mathcal F}$ at $(0,0)$.
\end{rem}

\begin{defi}
Consider the setting of Remark \ref{rem:order-formula-for-non-dicritical}. 
We define $v (D_j)$ as the {\it valence} of the irreducible component $D_j$ of the divisor of
$\pi$, i.e. $v (D_j)$ is the number of irreducible components of the divisor of $\pi$, different
than $D_j$ and intersecting $D_j$. We define the {\it non-dicritical valence} $v_{\overline{d}} (D_j)$
by considering just the invariant components of the divisor  intersecting $D_j$.  
\end{defi}

Formula (\ref{equ:funfor}) was generalized by Hertling to the general case.

\begin{pro}[\cite{Hertling:formules}]
\label{pro:hertling}
Consider the setting of Remark \ref{rem:order-formula-for-non-dicritical}. Then the formula
\begin{equation}
\label{equ:Hertling}
\nu_0 ({\mathcal F})  = \sum_{D_j} \sum_{P \in D_j} w(D_j) \kappa_P ({\mathcal F}', D_j)  +
\sum_{D_j\mathrm{\ non-invariant}} w(D_j)  (2 - v_{\overline{d}} (D_j)) -1
\end{equation} 
holds.
\end{pro}

\section{The Local Poincar\'{e} Problem}
We proceed to study the joint behaviour of the invariants associated to
a singular foliation ${\mathcal F}$ that has an invariant singular analytic branch $\gamma$
and the desingularization of $\gamma$.  Let us stress that we do not consider the
desingularization of ${\mathcal F}$.

  In order to do so, let us fix some
notation to be used in what follows. The curve
$\gamma:(\mathbb{C},0)\rightarrow (\mathbb{C}^2,0)$ is a separatrix of the
singular foliation $\mathcal{F}$ in $(\mathbb{C}^2,0)$. 
\begin{defi}
\label{def:genus}
The genus of $\gamma$ is the number of Puiseux
characteristic exponents in a reduced parametrization.
\end{defi}
We assume, henceforward, that $\gamma$ is singular, i.e. $g \geq 1$. 
Its desingularization requires at least $1$ blow-up (as a matter of fact, at least
$3$ but this is not relevant). We fix the sequence
{$\pi=\pi_{1}\circ \pi_2 \circ \dots \circ \pi_{k}$} of point blow-ups which
desingularizes $\gamma_0=\gamma$ and denote $P_0=(0,0)$ and for
$l=0,\dots, k-1$, $\gamma_{l+1}$ the strict transform of $\gamma_l$ by
{$\pi_{l+1}$,} $P_{l+1}$ the point defined
by $\gamma_{l+1}$ in {$D_{l+1} = \pi_{l+1}^{-1}(P_l)$.}  Set
$\mathcal{F}_0=\mathcal{F}$ and $\mathcal{F}_{l+1}$ the strict transform of
$\mathcal{F}_l$ by {$\pi_{l+1}$.}  Finally, for the sake of simplicity, we shall
denote ${\mathcal F}' = {\mathcal F}_k$ and
$\tilde{\pi}_{l+1}=\pi_{1}\circ \dots \circ \pi_{l+1}$. See Figure \ref{fig:notation} for a guide.

\begin{figure}[h!]
  \centering
    \begin{tikzpicture}[scale=0.6,every circle
      node/.style={draw,scale=0.6,fill=white},
     fixed point arithmetic={scale results=10^-6}]
\draw[line width=1.5pt,color=gray,domain=0:1] plot({\x*\x}, {\x*\x*\x});
\draw[line width=1.5pt,color=gray,domain=0:1] plot({\x*\x}, {-\x*\x*\x});
\node[anchor=south] at (1,0.9) {$\gamma_0=\gamma$};
\draw (-1,0) -- (1,0);
\draw (0,-1) -- (0,1);
\draw[<-] (1.5,0) -- (3,0);
\node[anchor=south] at (2.25,0) {$\pi_1$} ;
\draw[line width=1pt] (4,0) arc(0:20:4);
\draw[line width=1pt] (4,0) arc(0:-20:4) node[anchor=north] {$D_1$};
\draw[line width=1.5pt,color=gray,domain=0:1] plot({4+\x*\x},{\x});
\draw[line width=1.5pt,color=gray,domain=0:1] plot({4+\x*\x},{-\x});
\node[anchor=south] at (5,0.9) {$\gamma_1$};
\draw[line width=1pt] (8,0.5) arc(0:20:4)node[anchor=south] {$D_1$};
\draw[line width=1pt] (8,0.5) arc(0:-25:4);
\draw[line width=1pt] (9,-0.5) arc(90:115:4);
\draw[line width=1pt] (9,-0.5) arc(90:70:4) node[anchor=west] {$D_2$};
\draw[line width=1.5pt,color=gray,domain=0:2] plot({7.3+\x},{-1.2+\x});
\node[anchor=south] at (9.3,0.7) {$\gamma_2$};
\draw[<-] (5.5,0) -- (7,0);
\node[anchor=south] at (6.25,0) {$\pi_2$} ;
\draw[line width=1pt] (14,1) arc(0:20:4) node[anchor=south] {$D_1$};
\draw[line width=1pt] (14,1) arc(0:-20:4);

\draw[line width=1pt] (15,-1) arc(45:70:4);
\draw[line width=1pt] (15,-1) arc(45:35:4)node[anchor=north] {$D_3$};
\draw[line width=1pt] (16.8,-1) arc(90:115:4);
\draw[line width=1pt] (16.8,-1) arc(90:70:4) node[anchor=south] {$D_2$};
\draw[line width=1.5pt,color=gray,domain=0:2] plot({14.2+\x},{-1.2+\x});
\node[anchor=south] at (16.2,0.7) {$\gamma_3$};
\draw[<-] (11.5,0) -- (13,0);
\node[anchor=south] at (12.25,0) {$\pi_3$} ;
\end{tikzpicture}
  \caption{Notation for the blow-ups, exceptional divisors and strict transforms of the curve $\gamma\equiv(t^2, t^3)$. The points $P_i$ are the intersections of $\gamma_i$ with $D_i$.}
  \label{fig:notation}
\end{figure}
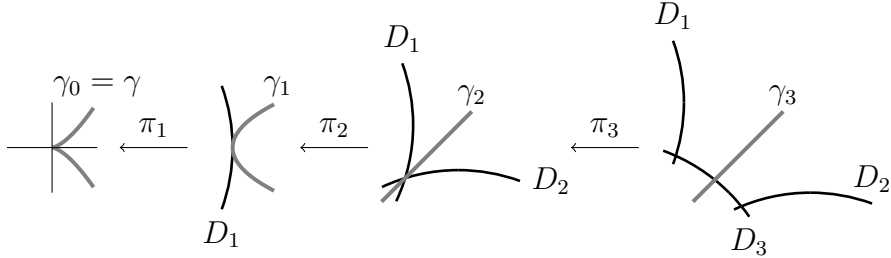

Let us remark that besides being regular we require $\gamma_k$ to intersect a
single irreducible component of $\pi^{-1}(0,0)$ and moreover that such
intersection is transversal.  Given an irreducible equation $f=0$ of $\gamma$,
this is equivalent to the function $f \circ \pi$ having only normal  crossings
singularities.
\begin{rem}
\label{rem:mult}
Let us apply Formula \eqref{equ:funfor} to the foliation $df=0$. Since $f \circ \pi =0$ has normal
crossings singularities, we obtain the equality
$w(D_k)= \nu_0 (df) +1 = \nu_0 (\gamma)$. Thus the last weight is the multiplicity at the origin of the
curve $\gamma$.
\end{rem}

Next, let us focus on the term
\[ \sum_{D_j} \sum_{P \in D_j} w(D_j) \kappa_P ({\mathcal F}', D_j)  \]
in Formula \eqref{equ:Hertling}.
\begin{defi}\label{def:inv-E}
  Given the exceptional divisor $E = \pi^{-1}(0,0)$, $\Inv(E)$ denotes the union of the irreducible
components of $E$ which are invariant for $\mathcal{F}'$.
We denote by ${\mathcal I}_{\mathcal F}$ (or ${\mathcal I}$ if ${\mathcal F}$ and $\gamma$ are
implicit) the connected components of $\Inv(E)$ and by $\tilde{\mathcal I}_{\mathcal F}$
or $\tilde{\mathcal I}$ the connected components of $\Inv(E)$ that do not contain
the last divisor $D_k$.
\end{defi}
\begin{rem}
Given a connected component $H$ of $\Inv(E)$ sometimes we consider it as a set of points and
so we write $P \in H$ and some other times as a set of divisors and then
we write $D \in H$.
\end{rem}
Figure \ref{fig:inv-E} illustrates this definition.
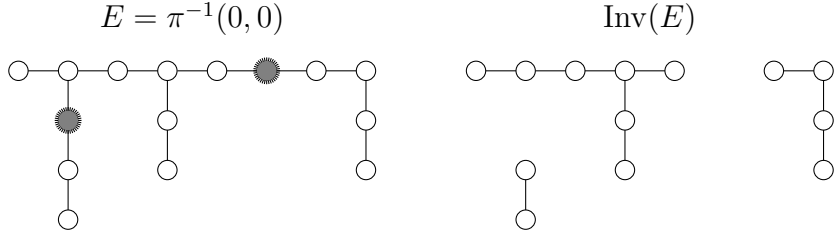
\begin{figure}[h!]
  \centering
  \begin{tabular}{cc}
   \begin{tikzpicture}[scale=0.66,every circle node/.style={draw,scale=0.66,fill=white},
  fixed point arithmetic={scale results=10^-6}]
  \draw[color=black] (0,0)node[circle,radius=1.5pt,color=black] {} -- (1,0)
  node(P1)[circle] {}-- (1,-1)node(Q1)[circle]{} -- (1,-2)node[circle]{} --
  (1,-3)node[circle]{};
  \draw[color=black] (P1) -- (2,0)node[circle]{} -- (3,0) node(P2)[circle] {} --
  (3, -1)node[circle]{} -- (3,-2)node[circle]{};
  \draw[color=black] (P2) -- (4,0)node[circle]{} -- (5,0)node(Q2)[circle]{}--(6,0)
  node[circle]{} --
  (7,0)node(P3)[circle]{}--(7,-1)node[circle]{}--(7,-2)node[circle]{};
  \fill[gray] (Q1) circle[radius=0.2];
  \fill[gray] (Q2) circle[radius=0.2];
  \foreach \i in {0,10,...,350}
  \tikzmath{\x=cos(\i)*0.27;\y=sin(\i)*0.27;}
  \draw[color=black] (Q1) -- +(\x,\y);
  \foreach \i in {0,10,...,350}
  \tikzmath{\x=cos(\i)*0.27;\y=sin(\i)*0.27;}
  \draw[color=black] (Q2) -- +(\x,\y);
  \draw[color=black] (3.5,0.5) node[anchor=south]{$E=\pi^{-1}(0,0)$};
\end{tikzpicture}\hspace*{10pt}
    &
      \hspace*{10pt}
      \begin{tikzpicture}[scale=0.66,every circle node/.style={draw,fill=white,scale=0.66},
  fixed point arithmetic={scale results=10^-6}]
  \draw[color=black] (0,0)node[circle,radius=1.5pt,color=black] {} -- (1,0)
  node(P1)[circle] {};
  \draw[color=black] (1,-2)node[circle]{} --
  (1,-3)node[circle]{};
  \draw[color=black] (P1) -- (2,0)node[circle]{} -- (3,0) node(P2)[circle] {} --
  (3, -1)node[circle]{} -- (3,-2)node[circle]{};
  \draw[color=black] (P2) -- (4,0)node[circle]{};
  \draw[color=black] (6,0) node[circle]{} -- (7,0)node(P3)[circle]{}
  --(7,-1)node[circle]{}--(7,-2)node[circle]{};
    \draw[color=black] (3.5,0.5) node[anchor=south]{$\Inv(E)$};
\end{tikzpicture}
  \end{tabular}
  \caption{For the exceptional divisor $E$ on the left (which has two dicritical irreducible components), the set $\Inv(E)$ is the one on the right, with three connected components.}
  \label{fig:inv-E}
\end{figure}
\begin{rem}
We have $\sharp ({\mathcal I} \setminus \tilde{\mathcal I}) \in \{1,0\}$ depending on whether or not
$D_k$ is invariant.
\end{rem}
\begin{defi}\label{def:weight-component}
Given a connected component $H$ of $\Inv(E)$, its \emph{weight} $w(H)$ is the minimum
of the weights of its irreducible components.
\end{defi}
\begin{pro}
\label{pro:toma}
Let $H$ be a connected component of $\Inv(E)$. Then
\begin{equation*}
  \sum_{\substack{D \in {H}\\P \in D}} \kappa_{P} ({\mathcal F}',D) w(D) \geq w(H) .
\end{equation*}
\end{pro}
\begin{proof}
It suffices to show that one of the indices in the sum is greater than $0$. Assume the contrary. Then, at every corner $P$ of $H$ the foliation ${\mathcal F}'$ is given, in local coordinates centered at $P$, by a vector field
\begin{equation*}
  x a(x,y) \frac{\partial}{\partial x} + y b(x,y) \frac{\partial}{\partial y}
\end{equation*}
where $xy=0$ is the local equation of $H$ at $P$. Since $\kappa_{P} ({\mathcal F}',x=0)$ and $\kappa_{P} ({\mathcal F}',y=0)$ are equal to $0$ by hypothesis, we get $\aleph_{P} ({\mathcal F}', x=0)=\aleph_{P} ({\mathcal F}', y=0)=1$. We deduce that $a(0,0) \neq 0 \neq b(0,0)$.
{By definition the Camacho-Sad index $\mathrm{CS} (P, x=0, {\mathcal F}')$
\cite{Cam-Sad}
of ${\mathcal F}'$ at $P$ along the invariant curve $x=0$ is the residue at $0$ of the diferential
form $\frac{a(0,y)}{y b(0,y)} dy$ and hence it is equal to $a(0,0)/b(0,0)$.
Analogously we obtain
$\mathrm{CS} (P, y=0, {\mathcal F}') =b(0,0)/a(0,0)$. So the Camacho-Sad indices at corners
satisfy
\begin{equation}
\label{equ:csc}
\mathrm{CS} (P, x=0, {\mathcal F}') \cdot \mathrm{CS} (P, y=0, {\mathcal F}') = 1 .
\end{equation}}
It is also known that the intersection form on divisors is negative
{(cf. \cite{Toma}) as is its
restriction to any connected component of $\Inv(E)$.
This property and Equation \eqref{equ:csc} on all corners allow us to find}
a point $Q \in H$ which belongs to a single irreducible component $D\in H$ for which $\mathrm{CS} (Q, D, {\mathcal F}') \neq 0$, following Toma \cite{Toma}. As $Q$ is a singular point of $\mathcal{F}'$, we get $\kappa_{Q} ({\mathcal F}', D) =  \aleph_{Q} ({\mathcal F}', D) \geq 1$, which provides the contradiction.
\end{proof}

We start now the joint study of the structure of the desingularization of $\gamma$ and Formula
\eqref{equ:Hertling} for $\mathcal{F}$. The reader should keep in mind Figure \ref{fig:first-graph},
which shows the desingularization graph for an invariant branch $\gamma$ in which possible
dicritical divisors of $\mathcal{F}'$ appear. Notice, however, that $\mathcal{F}'$ needs not have
any simple singularity in $E$ (not even the point at which $\gamma_k$ meets $E$ transversely).

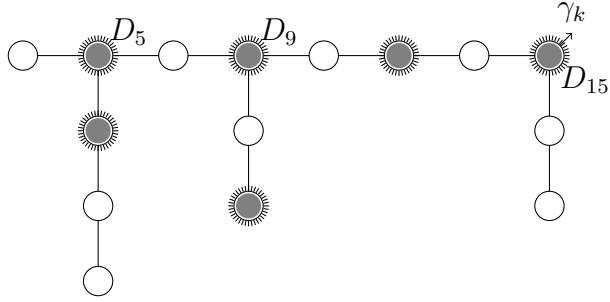
\begin{figure}[h!]
  \centering
  \begin{tikzpicture}[every circle node/.style={draw,fill=white},
  fixed point arithmetic={scale results=10^-6}]
  \draw[color=black] (0,0)node[circle,radius=3pt,color=black] {} -- (1,0)
  node(P1)[circle] {}-- (1,-1)node(Q1)[circle]{} -- (1,-2)node[circle]{} --
  (1,-3)node[circle]{};
  \draw (P1) -- (2,0)node[circle]{} -- (3,0) node(P2)[circle] {} --
  (3, -1)node[circle]{} -- (3,-2)node(Q3)[circle]{};
  \draw (P2) -- (4,0)node[circle]{} -- (5,0)node(Q2)[circle]{}--(6,0)
  node[circle]{} --
  (7,0)node(P3)[circle]{}--(7,-1)node[circle]{}--(7,-2)node[circle]{};
  \fill[gray] (Q1) circle[radius=0.1666];
  \fill[gray] (Q2) circle[radius=0.1666];
   \fill[gray] (Q3) circle[radius=0.1666];
   \fill[gray] (P1) circle[radius=0.1666];
   \fill[gray] (P2) circle[radius=0.1666];
   \fill[gray] (P3) circle[radius=0.1666];
  \draw[color=black,->] (P3) -- +(0.3,0.3) node[anchor=south] {$\gamma_k$};
  \foreach \i in {0,10,...,350}
  \tikzmath{\x=cos(\i)*0.27;\y=sin(\i)*0.27;}
  \draw[color=black] (Q1) -- +(\x,\y);
  \foreach \i in {0,10,...,350}
  \tikzmath{\x=cos(\i)*0.27;\y=sin(\i)*0.27;}
  \draw[color=black] (Q2) -- +(\x,\y);
   \foreach \i in {0,10,...,350}
  \tikzmath{\x=cos(\i)*0.27;\y=sin(\i)*0.27;}
  \draw[color=black] (Q3) -- +(\x,\y);
    \foreach \i in {0,10,...,350}
  \tikzmath{\x=cos(\i)*0.27;\y=sin(\i)*0.27;}
  \draw[color=black] (P1) -- +(\x,\y);
   \foreach \i in {0,10,...,350}
  \tikzmath{\x=cos(\i)*0.27;\y=sin(\i)*0.27;}
  \draw[color=black] (P2) -- +(\x,\y);
   \foreach \i in {0,10,...,350}
  \tikzmath{\x=cos(\i)*0.27;\y=sin(\i)*0.27;}
  \draw[color=black] (P3) -- +(\x,\y);
  \node[anchor=south west] at(P1) {$D_5$};
  \node[anchor=south west] at (P2) {$D_9$};
  \node[anchor=north west] at (P3) {$D_{15}$};
\end{tikzpicture}
\caption{Desingularization graph of $\gamma$ showing the dicritical divisors (in
  grey) of $\mathcal{F}'$.}
\label{fig:first-graph}
\end{figure}

Let us try to obtain a simple lower estimate for $\nu_{0} ({\mathcal F})$ from Hertling's formula
(\ref{equ:Hertling}).
\begin{defi}
We call $\NI(E)$ the set of non-invariant irreducible components $D$ of $E$ such
that there is $Q\in D$ with $\mathcal{F}'$
not transverse to $D$ at $Q$, i.e. $ \mathrm{tang}_Q ({\mathcal F}', D) \geq 1$.
\end{defi}
Hertling's formula implies the inequality:
\begin{equation*}
\begin{split}
  \nu_0 ({\mathcal F})  \geq  \sum_{D_j \not \subset \mathrm{Inv}(E)} \sum_{P \in D_j} w(D_j)
  \mathrm{tang}_P ({\mathcal F}', D_j) + \\
  \sum_{H \in {\mathcal I}} \sum_{D_j \in H} \sum_{P \in D_j} w(D_j) \kappa_P ({\mathcal F}', D_j)  +\\
\sum_{D_j\not\subset\mathrm{Inv}(E)} w(D_j)  (2 - v_{\overline{d}} (D_j)) -1
\end{split}
\end{equation*}
The contribution  of
$\sum_{D_j \in H} \sum_{P \in D_j} w(D_j) \kappa_P ({\mathcal F}', D_j)$ for $H \in \tilde{\mathcal I}$
is at least $w(H)$ by Proposition \ref{pro:toma}. Moreover we have
\begin{equation*}
\sum_{D_j \in H} \sum_{P \in D_j} w(D_j) \kappa_P ({\mathcal F}', D_j) \geq
w (D_k) = \nu_{0} (\gamma)
\end{equation*}
for $H \in {\mathcal I} \setminus \tilde{\mathcal I}$
since $\gamma_k$ intersects $D_k$ at a singular non-corner point if
$\mathcal{I}\setminus \tilde{\mathcal{I}}\neq \emptyset$.
We obtain
\begin{equation*}
\begin{split}
\nu_0 ({\mathcal F})  \geq  \sum_{D_j \not \subset \mathrm{Inv}(E)} \sum_{P \in D_j} w(D_j)
\mathrm{tang}_P ({\mathcal F}', D_j) +  \sum_{H \in \tilde{\mathcal I}} w(H) +  \\
\sharp ( {\mathcal I} \setminus \tilde{\mathcal I}) \nu_{0}(\gamma) +
\sum_{D_j \not \subset \mathrm{Inv}(E)} w(D_j)  (2 - v_{\overline{d}} (D_j)) -1 .
\end{split}
\end{equation*}
In order to make the estimate simpler
let us change slightly the definition of non-dicritical valence.
\begin{defi}
We define $\overline{v}_{\overline{d}} (D_k) = v_{\overline{d}} (D_k)+1$ and
$\overline{v}_{\overline{d}} (D_j) = v_{\overline{d}} (D_j)$ if $j < k$.
\end{defi}
The idea behind the definition is that for the last irreducible component of $E$
the curve $\gamma_k$ can be considered as an invariant divisor.
\begin{defi}\label{def:estimate}
We define
\begin{equation*}
\Lambda_{\mathcal{F}, {\mathcal I}} =  \sum_{H \in \tilde{\mathcal I}} w(H), \ \
\Lambda_{\mathcal{F}, {\mathcal N}}=
\sum_{D_j\not\subset\mathrm{Inv}(E)} w(D_j)  (2 - \overline{v}_{\overline{d}} (D_j))
\end{equation*}
and $\Lambda_{\mathcal{F}} =  \nu_{0}(\gamma) -1 +  \Lambda_{\mathcal{F}, {\mathcal I}}  +
\Lambda_{\mathcal{F}, {\mathcal N}}$.
\end{defi}
The term $\Lambda_{\mathcal{F}}$ is a lower bound for the multiplicity $\nu_0 ({\mathcal F})$.
It contains two parts, namely $\Lambda_{\mathcal{F}, {\mathcal I}}$ that provides a lower bound
for the contribution of invariant divisors in Hertling's formula and
$\Lambda_{\mathcal{F}, {\mathcal N}}$ that has an analogous role for non-invariant divisors.
The following proposition is a consequence of the above discussion.
\begin{pro} \label{pro:estimate}
The following sequence of inequalities holds:
\begin{equation*}
  \nu_0 ({\mathcal F}) \geq \Lambda_{{\mathcal F}} +
  \sum_{\substack{D\not \subset \Inv(E)\\P \in D}} w(D)
  \mathrm{tang}_P ({\mathcal F}_k, D) \geq  \Lambda_{{\mathcal F}}
+ \sum_{D \in \NI(E)} w(D) \geq \Lambda_{{\mathcal F}} .
\end{equation*}
\end{pro}

\begin{defi}\label{def:characteristicdivisor}
An irreducible component $D_l$ of the divisor is {\it a characteristic divisor} if
its valence is equal to $3$ or  it is the last divisor $D_k$
\end{defi}

The characteristic divisors in Figure (\ref{fig:first-graph}) are $D_5, D_9$ and $D_{15}$.

 Recall that we are {assuming} $\gamma$ to be singular of genus $g\geq 1$. Assume
its Puiseux expansion is
\begin{equation*}
  \label{eq:puiseux-expansion}
  \gamma(t) = \bigg(t^n, \sum_{i>n} c_i {t^{i}} \bigg)
\end{equation*}
and that its Puiseux characteristic exponents are $p_i/q_i$. Write
{$r_0=1$}, $q_1=r_1$ and $q_i=r_1\dots r_i$
(so that $n=r_1\dots r_g$).
There are as many characteristic divisors as Puiseux characteristic exponents.
Let $D_{a_1}, \hdots, D_{a_{g-1}}, D_{a_g}$ ($a_1 < a_2 < \hdots < a_g$) be the characteristic
divisors.   The weights satisfy $w(D_{a_j}) = q_j$.
 Roughly speaking, $q_j$ are the ``partial'' multiplicities
of the truncation of the Puiseux expansion of $\gamma$ to the corresponding order
(see Remark \ref{rem:mult}). 
\begin{rem}
Any divisor $D_j$ of $\pi$ satisfies $\overline{v}_{\overline{d}} (D_j) \leq  3$.
Notice that only characteristic divisors $D_j$ can satisfy
$\overline{v}_{\overline{d}} (D_j) = 3$. For instance if all irreducible components of $E$ are
invariant then the characteristic divisors $D_j$ are exactly those that satisfy
$\overline{v}_{\overline{d}} (D_j) = 3$.

Notice that all terms in the definition of $\Lambda_{\mathcal{F}, {\mathcal I}}$ are non-negative
whereas the term
$w(D_j)  (2 - \overline{v}_{\overline{d}} (D_j))$ in the sum defining
$\Lambda_{\mathcal{F}, {\mathcal N}}$ is negative if and only if
$\overline{v}_{\overline{d}} (D_j)=3$.
\end{rem}
\begin{defi}\label{def:baddivisor}
An irreducible component $D_l$ of the divisor is {\it a bad divisor} if
it is a non-invariant characteristic divisor and all adjacent irreducible components of
$E$ are invariant by ${\mathcal F}'$. A non-bad  divisor is called
{\it good}.
\end{defi}
 The divisor $D_5$  in Figure (\ref{fig:first-graph}) cannot be bad
  because it meets a dicritical component.
\begin{rem}
\label{rem:nobaddivisors}
We have  $\Lambda_{\mathcal{F}, {\mathcal N}} \geq 0$  if there are no bad divisors.
In this case,  $\Lambda_{\mathcal{F}, \mathcal{I}}\geq 0$   implies 
$\nu_0 ({\mathcal F}) \geq \nu_{0} (\gamma)-1$, which is the
lower bound for $\nu_0 ({\mathcal F})$ provided in
\cite{Camacho-Lins-Sad:topological} for the non-dicritical case.
\end{rem}

\begin{defi} \label{definition:irreducible-multiplicity} The
  {\it virtual multiplicity} $\mu(\gamma)$ of $\gamma$ is
   $q_{g-1}$. 
\end{defi}

\begin{defi}
A \emph{configuration of bad divisors} is a     sequence   
$j_1< \dots< j_p$ with $j_i\in \left\{ 1, \dots, g \right\}$.
We say that it is a configuration of bad divisors for $\gamma$ and
${\mathcal F}$ if the bad divisors of ${\mathcal F}$ are
$D_{b_1}, D_{b_2}, \hdots, D_{b_p}$ where $b_i = a_{j_i}$.
\end{defi}

\begin{figure}[h!]
  \centering
  \begin{tabular}{cc}
            \begin{tikzpicture}[scale=0.66,every circle node/.style={draw,scale=0.66,fill=white},
  fixed point arithmetic={scale results=10^-6}]
  \draw[color=black] (0,0)node[circle,radius=1.5pt,color=black] {} -- (1,0)
  node(P1)[circle] {}-- (1,-1)node(Q1)[circle]{} -- (1,-2)node[circle]{} --
  (1,-3)node[circle]{};
  \draw (P1) -- (2,0)node[circle]{} -- (3,0) node(P2)[circle] {} --
  (3, -1)node[circle]{} -- (3,-2)node[circle]{};
  \draw (P2) -- (4,0)node[circle]{} -- (5,0)node(Q2)[circle]{}--(6,0)
  node[circle]{} --
  (7,0)node(P3)[circle]{}--(7,-1)node[circle]{}--(7,-2)node[circle]{};
  \fill (P2) circle[radius=0.2,black];
  \fill (P3) circle[radius=0.2,black];
  \fill[gray] (Q1) circle[radius=0.2];
  \fill[gray] (Q2) circle[radius=0.2];
  \draw[color=black,->] (P3) -- +(0.3,0.3);
  \foreach \i in {0,10,...,350}
  \tikzmath{\x=cos(\i)*0.27;\y=sin(\i)*0.27;}
  \draw[color=black] (Q1) -- +(\x,\y);
  \foreach \i in {0,10,...,350}
  \tikzmath{\x=cos(\i)*0.27;\y=sin(\i)*0.27;}
  \draw[color=black] (Q2) -- +(\x,\y);
  \node[anchor=south] at(P1) {$D_5$};
  \node[anchor=south] at (P2) {$D_9$};
  \node[anchor=north west] at (P3) {$D_{15}$};
\end{tikzpicture}
    &
        \begin{tikzpicture}[scale=0.66,every circle node/.style={draw,fill=white,scale=0.66},
  fixed point arithmetic={scale results=10^-6}]
  \draw[color=black] (0,0)node[circle,radius=1.5pt,color=black] {} -- (1,0)
  node(P1)[circle] {}-- (1,-1)node(Q1)[circle]{} -- (1,-2)node[circle]{} --
  (1,-3)node[circle]{};
  \draw (P1) -- (2,0)node[circle]{} -- (3,0) node(P2)[circle] {} --
  (3, -1)node[circle]{} -- (3,-2)node[circle]{};
  \draw (P2) -- (4,0)node[circle]{} -- (5,0)node(Q2)[circle]{}--(6,0)
  node[circle]{} --
  (7,0)node(P3)[circle]{}--(7,-1)node[circle]{}--(7,-2)node[circle]{};
  \fill (P3) circle[radius=0.2,black];
  \fill[gray] (P1) circle[radius=0.2];
  \fill[gray] (Q1) circle[radius=0.2];
  \fill[gray] (Q2) circle[radius=0.2];
  \draw[color=black,->] (P3) -- +(0.3,0.3);
   \foreach \i in {0,10,...,350}
  \tikzmath{\x=cos(\i)*0.27;\y=sin(\i)*0.27;}
  \draw[color=black] (P1) -- +(\x,\y);
  \foreach \i in {0,10,...,350}
  \tikzmath{\x=cos(\i)*0.27;\y=sin(\i)*0.27;}
  \draw[color=black] (Q1) -- +(\x,\y);
  \foreach \i in {0,10,...,350}
  \tikzmath{\x=cos(\i)*0.27;\y=sin(\i)*0.27;}
  \draw[color=black] (Q2) -- +(\x,\y);
  \node[anchor=south] at(P1) {$D_5$};
  \node[anchor=south] at (P2) {$D_9$};
  \node[anchor=north west] at (P3) {$D_{15}$};
\end{tikzpicture}
  \end{tabular}
  \caption{Two possible configurations of bad divisors (in black)
  for $\gamma$ and $\mathcal{F}$.}
  \label{fig:bad-divisors-configs}
\end{figure}
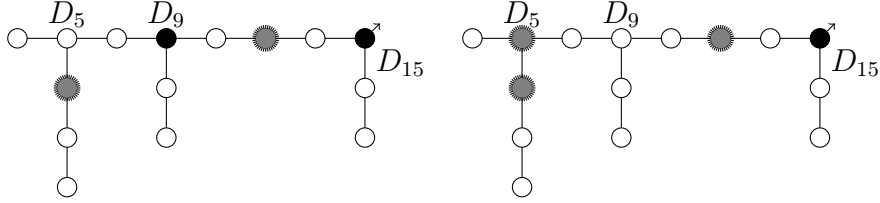
\begin{rem}
\label{rem:noconsbad}
Two characteristic divisors $D_{a_j}$ and $D_{a_{j+1}}$ such that
$D_{a_j} \cap D_{a_{j+1}} \neq \emptyset$ can not be simultaneously bad.
\end{rem}
The next proposition illustrates that once a configuration of bad divisors is fixed, the worst lower
estimate $\Lambda_{\mathcal{F}}$ is obtained when all good divisors are invariant.
\begin{pro}
\label{pro:minlam}
Let ${\mathcal F}, {\mathcal G}$ be germs of singular holomorphic foliations, defined in a
neighborhood of the origin in ${\mathbb C}^{2}$,
with a common invariant curve $\gamma$. Suppose ${\mathcal F}, {\mathcal G}$
share a configuration of bad divisors for $\gamma$.
Suppose that all non-invariant divisors of ${\mathcal F}$ are bad and that
${\mathcal G}$ has more non-invariant divisors than
${\mathcal F}$. Then
$\Lambda_{\mathcal{G}} > \Lambda_{\mathcal{F}}$.
\end{pro}
\begin{proof}
  Notice that all divisors adjacent to bad divisors are invariant
  for both ${\mathcal F}$ and ${\mathcal G}$. Thus the
  contribution of the term
  $w(D_j) (2 - \overline{v}_{\overline{d}} (D_j))$ of a bad divisor $D_j$ to
  $\Lambda$ is the same for ${\mathcal F}$ and ${\mathcal G}$.  Moreover, we
  also deduce
\begin{equation*}
\sum_{\substack{J \in \tilde{\mathcal I}_{\mathcal G}\\J \subset H}} w(J) \geq 0 \ \ (\mathrm{resp.} \
\sum_{\substack{J \in \tilde{\mathcal I}_{\mathcal G}\\J \subset H}} w(J) \geq w(H))
\end{equation*}
for any $H \in {\mathcal I}_{\mathcal F}$ (resp. $H \in \tilde{\mathcal I}_{\mathcal F}$). We deduce
$\Lambda_{{\mathcal G}, {\mathcal I}} \geq \Lambda_{{\mathcal F}, {\mathcal I}}$.

Since any good divisor $D_j$ satisfies
$\overline{v}_{\overline{d}} (D_j) \leq 2$, we conclude that the inequality
$\Lambda_{{\mathcal G}, {\mathcal N}} - \Lambda_{{\mathcal F}, {\mathcal N}}
\geq 0$ holds.

Let $D_j$ be a divisor that is non-${\mathcal G}$-invariant but is invariant for ${\mathcal F}$.
We can assume $\overline{v}_{\overline{d}} (D_j) = 2$
(where the valence is with respect to $\mathcal{G}$) since otherwise
it is clear that $\Lambda_{{\mathcal G}, {\mathcal I}} > \Lambda_{{\mathcal F}, {\mathcal I}}$
and then $\Lambda_{\mathcal{G}} > \Lambda_{\mathcal{F}}$.
There are two possibilities, namely $D_j=D_k$ and ${v}_{\overline{d}} (D_j) = 1$ or $j<k$ and
${v}_{\overline{d}} (D_j) = 2$.
In the former case $D_j$ is contained in the unique element $H$ of
${\mathcal I}_{\mathcal F} \setminus \tilde{\mathcal I}_{\mathcal F} $ and
there is a unique $J \in \tilde{I}_{\mathcal G}$ that has a divisor adjacent to
$D_j$. Since it satisfies $J \subset H$ and $w(J) >0$, it follows that
$\Lambda_{{\mathcal G}, {\mathcal I}} > \Lambda_{{\mathcal F}, {\mathcal I}}$ and
then  $\Lambda_{\mathcal{G}} > \Lambda_{\mathcal{F}}$, as desired.
In the latter case let $H$ be the element of ${\mathcal I}_{\mathcal F}$ containing $D_j$.
We denote by $H_1$ and $H_2$ the elements of ${\mathcal I}_{\mathcal G}$ containing divisors
adjacent to $D_j$. They satisfy $H_1 \cup H_2 \subset H$.
Either $H_1$ or $H_2$ belongs to
$\tilde{\mathcal I}_{\mathcal G}$ and thus $\sum_{H_j \in \tilde{\mathcal I}_{\mathcal G}} w(H_j) >0$.
Moreover since $w(H_1) \geq w(H) \leq w(H_2)$ we deduce
$\Lambda_{{\mathcal G}, {\mathcal I}} - \Lambda_{{\mathcal F}, {\mathcal I}} > 0$ and then
$\Lambda_{{\mathcal G}} > \Lambda_{{\mathcal F}}$.
\end{proof}

We can now compute lower bounds for $\Lambda_{\mathcal{F}}$ in terms of its configuration
of bad divisors.

\begin{defi}
  Fix a configuration of bad divisors ${\mathcal B}= \{ j_1, \hdots,
  j_p\}$. Given $1 < l \leq p$, we call $F_{l}$ the connected component of
  $E \setminus \cup_{q=1}^{p} D_{{b_q}}$ that intersects both $D_{{b_{l-1}}}$
  and $D_{{b_l}}$. For $l=1$, $F_{1}$ is the connected component of
  $E \setminus \cup_{q=1}^{p} D_{{b_q}}$ that contains the first divisor
  {$\pi_{1}^{-1}(0,0)$.} Given $1 \leq l \leq p$ we denote by $C_l$ the other
  connected component of $E \setminus \cup_{r=1}^{p} D_{{b_r}}$ that
  intersects $D_{{b_l}}$. The components $F_l$ are called the \emph{free
    components} of $E$ and the $C_l$ are the \emph{clamped} components.
\end{defi}

Figure \ref{fig:connected-components-first-graph} provides an example of the previous definition.
\begin{rem}
  By construction, the clamped components $C_1, \hdots, C_p$ are non-empty.
  Moreover, the free components $F_1, \hdots F_p$ are also non-empty.  Indeed if
  $F_l = \emptyset$ then $D_{b_{l-1}}$ and $D_{b_l}$ intersect in $E$ and
  this contradicts that both divisors are bad by Remark \ref{rem:noconsbad}.
\end{rem}
\begin{figure}[h!]
  \centering
  \begin{tikzpicture}[every circle node/.style={draw,fill=white},
  fixed point arithmetic={scale results=10^-6}]
  \draw[color=black] (0,0)node[circle,radius=3pt,color=black] {} --
  (1,0) node(F1)[circle]{} -- (2,0);
  \draw[color=black] (1,0)node[circle]{} -- (1,-1) node[circle]{};
  \draw[color=black] (1,-1)node(C11)[circle]{} -- (1,-2)node(C12)[circle]{} --
  (1,-3)node(C13)[circle]{};
  \draw (2,0)node(F21)[circle]{};
  \draw (3, -1)node(C21)[circle]{} -- (3,-2)node(C22)[circle]{};
  \draw (4,0)node(F31)[circle]{} -- (5,0)node(F32)[circle]{}--(6,0)
  node(F33)[circle]{};
  \draw (7,-1)node(C31)[circle]{}--(7,-2)node(C32)[circle]{};
  \fill[gray] (F1) circle[radius=0.1666];
  \fill[gray] (C11) circle[radius=0.1666];
   \fill[gray] (C22) circle[radius=0.1666];
  \fill[gray] (F32) circle[radius=0.1666];
   \foreach \i in {0,10,...,350}
  \tikzmath{\x=cos(\i)*0.27;\y=sin(\i)*0.27;}
  \draw[color=black] (F1) -- +(\x,\y);
  \foreach \i in {0,10,...,350}
  \tikzmath{\x=cos(\i)*0.27;\y=sin(\i)*0.27;}
  \draw[color=black] (C11) -- +(\x,\y);
   \foreach \i in {0,10,...,350}
  \tikzmath{\x=cos(\i)*0.27;\y=sin(\i)*0.27;}
  \draw[color=black] (C22) -- +(\x,\y);
  \foreach \i in {0,10,...,350}
  \tikzmath{\x=cos(\i)*0.27;\y=sin(\i)*0.27;}
  \draw[color=black] (F32) -- +(\x,\y);
  \draw[rounded corners, dashed](-0.33,0) -- (-0.33,0.33) -- (2.33,0.33) --
  (2.33,-0.33) -- (1.33,-0.33) -- (1.33,-3.33) -- (0.66, -3.33) -- (0.66, -0.33)
  -- (-0.33,-0.33) -- (-0.33,0);
  \path (F1)--+(0,0.6) node{$F_1$};
  \path (C22)--+(0.65,0.5) node{$C_1$};
  \path (F32)--+(0,0.6) node{$F_2$};
  \path (C32)--+(0.65,0.5) node{$C_2$};
  \node[draw,dashed,fit=(C21) (C22), rounded corners]{};
  \node[draw,dashed,fit=(F31)(F32)(F33), rounded corners]{};
  \node[draw,dashed,fit=(C31)(C32), rounded corners]{};
\end{tikzpicture}
\caption{Connected components $F_l$ and $C_l$ for the graph in Figure \ref{fig:first-graph}, for the configuration of bad divisors $(9, 15)$.}
\label{fig:connected-components-first-graph}
\end{figure}
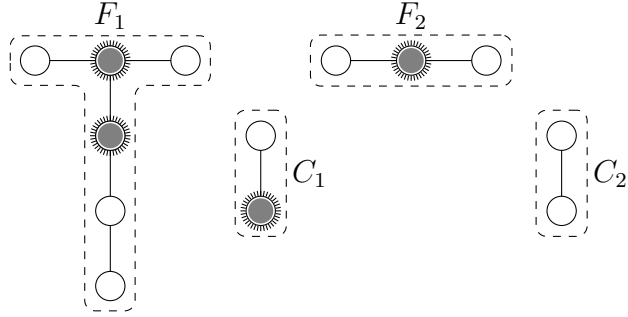

\begin{rem}
The weights of $C_l$ are exactly  $w(C_l)= q_{j_{l}-1}$  
where as above,  $q_{j_{l}-1}$  is the {$j_l-1$-th}
``partial multiplicity'' of $\gamma$. On the other hand,
 $w(F_{l}) \geq  q_{j_{l-1}}$  and the equality happens
if and only if $D_{a_{j_{l-1}+1}}$ is good or
$P_{b_{l-1}+1}$ is a non-corner point.
\end{rem}

 \begin{defi}\label{def:bad-foliation-for-gamma}
   A foliation $\mathcal{G}$ whose only non-invariant divisors are all bad divisors is called a
   \emph{bad foliation for} $\gamma$.  Given its configuration of bad divisors
   ${\mathcal B}$ we say that $\mathcal{G}$ is a bad foliation for $\gamma$ and
   $\mathcal{B}$.
 \end{defi}
\begin{pro}
\label{pro:calam}
Consider an irreducible curve  $\gamma$ invariant by a germ of foliation ${\mathcal F}$.
Let ${\mathcal B}= \{ j_1,\dots, j_p\}$ be a configuration of bad divisors for ${\mathcal F}$
and $\gamma$ and
set $j_0=0$. Let $\mathcal{G}$ be a bad foliation for $\gamma$ and $\mathcal{B}$.
Then: 
\begin{equation}
\label{equ:calam}
  \Lambda_{{\mathcal G}}= q_{g}  -1 + \sum_{l=1}^{p} \left(
  - q_{j_l} + q_{j_{l}-1}+ w(F_l ) \right) \geq
q_g - q_{j_p} +  \sum_{l=1}^{p} q_{j_{l}-1} .
 \end{equation}
In particular, for $\mathcal{F}$, we get
\begin{equation}
\label{equ:calam2}
\nu_{0} ({\mathcal F}) \geq \Lambda_{{\mathcal F}} \geq
q_g - q_{j_p} +  \sum_{l=1}^{p} q_{j_{l}-1}.
\end{equation} 
\end{pro}
\begin{proof}
  Since $\nu_{0} ({\mathcal F})  \geq \Lambda_{{\mathcal F}} \geq \Lambda_{{\mathcal G}}$
by Propositions \ref{pro:estimate} and \ref{pro:minlam},
it sufffices to prove (\ref{equ:calam}). It is clear that 
$\Lambda_{{\mathcal G}} = w(D_k) -1= q_g -1$ 
if ${\mathcal B} = \emptyset$ by definition,
so we assume ${\mathcal B} \neq \emptyset$ from now on.
By Definition \ref{def:estimate} we have
\begin{equation*} 
  \Lambda_{{\mathcal G}}= q_g -1 -   \sum_{l=1}^{p}  {w(D_{b_l})}
  + \sum_{l=1}^{p} w(F_l) + \sum_{l=1}^{p} w({C_l}) 
\end{equation*}
From the equalities  $w(D_{b_l}) = q_{j_l}$   and
 $w(C_l) = q_{j_l -1}$  we obtain the equality in \eqref{equ:calam}. Using now
 $w(F_l) \geq q_{j_{l-1}}$,  we deduce
\begin{equation*} 
  \begin{split}
\Lambda_{{\mathcal G}} \geq  q_g -1 +   \sum_{l=1}^{p} \left( - q_{j_l}
+  q_{j_{l-1}} \right) + \sum_{l=1}^{p} q_{j_l -1}     \\
=q_g -1 - q_{j_p} +  1  + \sum_{l=1}^{p} q_{j_l -1}.
\end{split} 
\end{equation*}
\end{proof}
From Propositions \ref{pro:estimate} and \ref{pro:calam} we obtain
\begin{cor}
\label{cor:formula2}
The following sequence of inequalities holds: 
\begin{equation}
\label{equ:formula2}
\begin{split}
  \nu_0 ({\mathcal F}) \geq \Lambda_{{\mathcal F}} +
  \sum_{\substack{D\not \subset \Inv(E)\\P \in D}} w(D)
  \mathrm{tang}_P ({\mathcal F}_k, D) \geq \\
  q_g - q_{j_p} +  \sum_{l=1}^{p} q_{j_{l}-1} + \sum_{D \in \NI(E)} w(D).
\end{split}
\end{equation} 
\end{cor}
\begin{proof}[Proof of the Main Theorem]
  The result is clear if the configuration of bad divisors of $\gamma$ is empty since 
$\nu_{0} ({\mathcal F}) \geq q_g -1$  by Remark \ref{rem:nobaddivisors}.
We assume, from now on, that it is not empty.

  If the last divisor $D_k$ is not bad, then Proposition \ref{pro:calam} implies
\begin{equation*} 
  \nu_0 ({\mathcal F}) \geq q_g - q_{j_p}  + q_{j_p -1} . 
\end{equation*}
From which {we obtain}
\begin{equation*} 
\frac{\nu_0 ({\mathcal F})}{q_g} > 1 - \frac{1}{2} = \frac{1}{2} . 
\end{equation*}
And, since  $q_g/q_{g-1}= r_g\geq 2$, 
we get   $\nu_0 ({\mathcal F}) >  q_{g-1}$. 

Finally, if $D_k$ is a bad divisor, then 
$\nu_0 ({\mathcal F}) \geq q_g - q_{g} + q_{g -1} = q_{g -1}$ 
by Proposition \ref{pro:calam}.
\end{proof}
\begin{rem}
As an application of our theorem, we obtain useful lower bonds for the multiplicity of foliations
preserving simple irreducible curves. For example,
let $\gamma$ be the curve with Puiseux parametrization
$t \mapsto (t^{2n}, t^{2n+2} + t^{2n+3})$  ($n \geq 2$). It has two Puiseux characteristic
exponents and $g=2$, $r_1 = n$, $r_2=2$.
Every vector field $X \in {\mathfrak X} ({\mathbb C}^{2},0)$ that preserves $\gamma$
satisfies  $\nu_0 (X) \geq \frac{\nu_0 (\gamma)}{2}=n$ by  the Main Theorem. 
The cusp of parametrization
$t \mapsto (t^{n}, t^{n+1})$ is an approximation of $\gamma$ obtained by removing
the higher order term and whose multiplicity is a lower bound for $\nu_0 (X)$.
\end{rem}
\begin{cor}
If an analytic branch  $\gamma$ has genus $g\geq 2$ and is invariant for $\mathcal{F}$,
then $\nu_0 ({\mathcal F}) \geq 2$.
\end{cor}

\begin{teo} \label{teo:bound-H-all-denoms}
  Let $\gamma$ and $\mathcal{F}$ be as above with $g\geq 1$. Let $P_{k-1}$ be the last center of $\pi$. The following alternative holds:
  \begin{itemize}
  \item Either  $\nu_0(\gamma)= q_g \leq 2 q_{g-1} (r_{g} -1) \leq 2\nu_0(\mathcal{F})$,
   or
  \item $P_{k-1}$ is a radial singularity of $\mathcal{F}_{k-1}$ and $D_k$ is a bad divisor.
  \end{itemize}
Moreover, if we are not in the latter case and
the equality $\nu_0(\gamma) = 2\nu_0(\mathcal{F})$
holds, then $\nu_0(\gamma)=2$ and the whole exceptional divisor $E$ is ${\mathcal F}'$-invariant.
\end{teo}
\begin{proof}
  We consider the same cases as in the proof of  the Main Theorem. 

Assume, first, that $\gamma$ contains no bad divisors for $\mathcal{F}$. {We deduce}
\begin{equation*} 
\nu_{0} ({\mathcal F}) \geq \Lambda_{{\mathcal F}} \geq  q_g -1 \geq
q_{g-1} (r_g -1) \geq \frac{q_g}{2} = \frac{\nu_0(\gamma)}{2} 
\end{equation*}
by Proposition \ref{pro:calam}. The third inequality is strict
unless $g=1$ and the fourth inequality is strict unless $r_g=2$.
Therefore, if  $\nu_{0} ({\mathcal F}) = 2^{-1} q_g$ 
holds then  $\Lambda_{{\mathcal F}} =  q_1  -1$ 
and this implies that all the divisors $D_1,\dots, D_k$ are ${\mathcal F}_k$-invariant
by Proposition \ref{pro:minlam}.

Assume now that the configuration of bad divisors of $\gamma$ is non-empty but that $D_k$ is good.
This implies $g >1$. Then
\begin{equation*} 
\nu_0 ({\mathcal F}) > q_g - q_{g-1} = q_{g-1}(r_g -1) \geq \frac{\nu_0(\gamma)}{2} 
\end{equation*}
by Formula \eqref{equ:calam2}.

Finally, assume that $D_k$ is a bad divisor but that {$P_{k-1}$}
is not a radial singularity of ${\mathcal F}_{k-1}$. Then
$\sum_{P \in D_k} \mathrm{tang}_P ({\mathcal F}_k, D_k) \geq 1$ and
by Corollary \ref{cor:formula2}  we get
\begin{equation*} 
    \nu_0 ({\mathcal F}) \geq  \Lambda_{{\mathcal F}} + w(D_k)
  \geq q_g - q_{g} + q_{g -1} + w(D_k) \geq q_g  +  q_{g -1} > \nu_0(\gamma)  
\end{equation*}
and the proof is complete.
\end{proof}
\begin{rem}
As the previous result shows, the worst possible situation for our lower bounds of $\nu_0 ({\mathcal F})$ in terms of {$\nu_0 (\gamma)$} happens when
\begin{equation}
\label{equ:badcase}
\nu_{0} ({\mathcal F}) < \frac{{\nu_0 (\gamma)}}{2} \ \ \mathrm{and} \ \
\nu_{0} ({\mathcal F}) =   \mu   (\gamma)  = q_{g-1} . 
\end{equation}
In this case, $D_{k}$ is necessarily a bad divisor and ${\mathcal F}_{k-1}$ has a radial singularity at
$P_{k-1}$. Moreover, $D_k$ is the unique bad divisor for ${\mathcal F}$ since otherwise we have
 $\nu_0 ({\mathcal F}) \geq \Lambda_{{\mathcal F}}  > q_{g-1}$  
by  Proposition \ref{pro:calam}.
A bad foliation ${\mathcal G}$ whose unique bad divisor is $D_k$ satisfies 
$\Lambda_{{\mathcal G}} = q_{g-1}$.  Therefore
the inequality $ \Lambda_{{\mathcal F}} \geq \Lambda_{{\mathcal G}}$ is strict if
there exists a divisor, other than $D_k$, that is non-invariant by Proposition \ref{pro:minlam}:
in this case we obtain again  $\Lambda_{{\mathcal F}} > q_{g-1}$. 

In summary, {property} (\ref{equ:badcase}) happens only if all the exceptional divisors of the
desingularization of $\gamma$ except the last one $D_k$ are invariant, $D_k$ is non-invariant and
$P_{k-1}$ is a radial singularity of ${\mathcal F}_{k-1}$.
\end{rem}
\begin{rem} 
\label{rem:sharp}
Let us explain how to construct examples of pairs $({\mathcal F}, \gamma)$ of
foliations and invariant branches such that $\nu_{0} ({\mathcal F}) = \mu (\gamma)$.
Let $\gamma'$ be a germ of irreducible analytic curve such that $g \geq 2$ and fix
its desingularization $\pi$.  Consider an irreducible analytic curve $\xi$ such
that its strict transform by $\pi$ intersects transversally in a non-corner point the
unique irreducible component $D_j$ of the last clamped component $C_p$ whose weight
is equal to $q_{g-1}$.  This property implies $\nu_0 (\xi) = q_{g-1}$.  Let $\ell$
a line whose tangent cone is different than the tangent cone of $\gamma'$.  We denote
by $f=0$ and $L=0$ irreducible equations of $\xi$ and $\ell$ respectively.  Let
$d$ (resp. $d'$) be the order of $f \circ \pi$ (resp. $L \circ \pi$) along $D_k$.  We
denote $a = \mathrm{lcm}(d,d') / d$ and $c= \mathrm{lcm}(d,d') / d'$ and consider the
foliation ${\mathcal F}$ given by its meromorphic first integral $F= f^{a}/L^{c}$.
The function $F \circ \tilde{\pi}_{j-1}$ is of the form $\tilde{f}^{a}/M^{c'}$ in the
neighborhood of $P_{j-1}$ where $\tilde{f}=0$ is the strict transform of $f=0$ and
$M=0$ is a local equation of the divisor $D_{j-1}$. The curves $\tilde{f}=0$ and
$D_{j-1}$ are smooth and transversal at $P_{j-1}$ and hence $\tilde{f}^{a}/M^{c'}$ is
equal to $y_1^{a}/x_1^{c'}$ in some coordinates centered at $P_{j-1}$.  Moreover, we
have $c'>0$ since otherwise the order of $F \circ \pi$ along $D_k$ is greater than
$0$ contradicting the choice of $F$. Since all the points $P_{j}, \hdots, P_{k-1}$
are corners, there is just one non-invariant divisor $D$ among
$D_{j-1}, \hdots, D_k$, it is characterized by $F \circ \pi$ having order $0$ along
$D$ and thus $D=D_k$. Moreover since $\tilde{f}^{a}/M^{c'} = y_1^{a}/x_1^{c'}$, the
singular point $P_{k-1}$ is radial and $D_k$ is a bad divisor. Let $\gamma$ be one of
the irreducible components of $F=1$ and denote $ \omega = a L df - c f dL$.  Notice
that $\gamma$ and $\gamma'$ have the same desingularization. We have
$$ \nu_{0} ({\mathcal F}) = \nu_0 (\omega=0) = \nu_0 (\xi)= q_{g-1} = \mu (\gamma) .$$
The above equality implies that $D_k$ is the unique bad divisor of $\gamma$ by
Equation (\ref{equ:calam2}).  Moreover, $D_k$ is the unique non-invariant divisor of
$\pi^{-1}(0,0)$ by Proposition \ref{pro:minlam}.

As an illustration of this method,
consider $f= y^{2} - x^{3}$, $L=x$, $a=4$, $c=13$. Let  $\gamma$ be the
curve with Puiseux parametrization $(t^{4}, t^{6} \sqrt{1+t})$.
It satisfies $g=2$, $q_1=2$ , $q_2=4$ and $F (t^{4}, t^{6} \sqrt{1+t}) \equiv 1$.
The foliation with first integral $F$ is given by the vector field
$$
X= 8 x y \frac{\partial}{\partial x} + (13 y^{2} - x^{3})  \frac{\partial}{\partial y},
$$
for which $\gamma$ is invariant and satisfies $\nu_{0} (X) = 2 = q_1 = \mu (\gamma)$.
\end{rem}

Finally, an immediate consequence of the Main Theorem is that the genus of an
invariant curve is bounded by a function of the multiplicity of the foliation.
\begin{cor}
Let $\gamma$ be an irreducible curve of a germ of singular holomorphic foliation ${\mathcal F}$
defined in the neighborhood of $(0,0)$ in ${\mathbb C}^{2}$. Then we get
$\nu_{0} ({\mathcal F}) \geq 2^{g-1}$ and $g \leq 1 + \log_{2} \nu_{0} ({\mathcal F})$.
\end{cor}

\bibliography{sb}
\end{document}